\documentclass[a4paper]{amsart}
\usepackage{amssymb, amsmath, amsthm}
\usepackage{stmaryrd, esint}
\usepackage{xcolor}
\usepackage{hyperref}

\newtheorem{theorem}{Theorem}[section]
\newtheorem{lemma}[theorem]{Lemma}
\newtheorem{proposition}[theorem]{Proposition}
\newtheorem{corollary}[theorem]{Corollary}

\theoremstyle{definition}
\newtheorem{definition}[theorem]{Definition}

\theoremstyle{remark}
\newtheorem{remark}[theorem]{Remark}

\numberwithin{equation}{section}

\newcommand{\R}{\mathbb{R}}

\newcommand{\C}{\mathbb{C}}
\newcommand{\N}{\mathbb{N}}

\newcommand{\mint}{\mathop{\int\hspace{-1.05em}{\--}}\nolimits}

\DeclareMathOperator{\dist}{dist}
\DeclareMathOperator{\spt}{spt}
\DeclareMathOperator{\Ind}{Ind}
\DeclareMathOperator{\Nul}{Nul}

\title[Index estimates]{Index estimates for sequences of harmonic maps}
\author{Jonas Hirsch} 
\address[J.~Hirsch]{Mathematisches Institut\\ Universit\"at Leipzig\\ Augustusplatz 10\\ 04109 Leipzig\\ Germany}
\email{hirsch@math.uni-leipzig.de}

\author{Tobias Lamm}
\address[T.~Lamm]{Institute for Analysis\\
Karlsruhe Institute of Technology \\ Englerstr. 2\\
76131 Karlsruhe\\ Germany}
\email{tobias.lamm@kit.edu}

\date{\today}
\thanks{We thank the referee for the careful reading of the paper and the helpful comments.}
\begin{document}
\begin{abstract}
In this paper we study upper and lower bounds of the index and the nullity for sequences of harmonic maps with uniformly bounded Dirichlet energy from a two-dimensional Riemann surface into a compact target manifold. The main difficulty stems from the fact that in the limit the sequence can develop finitely many bubbles. We obtain the index bounds by studying the limiting behavior of sequences of eigenfunctions of the linearized operator and the key novelty of the present paper is that we diagonalize the index form of the Dirichlet energy with respect to a bilinear form which varies with the sequence of harmonic maps and which helps us to show the convergence of the sequence of eigenfunctions on the weak limit, the bubbles and the intermediate neck regions. 

Finally, we sketch how to modify our arguments in order to also cover the more general case of sequences of critical points of two-dimensional conformally invariant variational problems. 
\end{abstract}
\maketitle

\section{Introduction}
For a closed Riemann surface $(M,g)$ and a closed Riemannian manifold $(N,h)$, which we assume to be isometrically embedded into some euclidean space $\R^m$, we define the Dirichlet energy for maps $u\in W^{1,2}(M,N)$ by
\[
E(u)=\frac12 \int_M |\nabla u|^2\, dv_g.
\]
Critical points of $E$ are called harmonic maps and they are solutions of the elliptic partial differential equation
\[
-\Delta u = A(u)(\nabla u, \nabla u),
\]
where $A$ is the second fundamental form of the embedding of $N \hookrightarrow \R^m$. 

In this paper we study sequences of harmonic maps $u_k\in C^\infty(M,N)$, $k\in \N$, with uniformly bounded Dirichlet energy. By now it is well-known that such a sequence has a weak limit $u_0\in C^\infty(M,N)$  which is again a harmonic map and that we actually have local smooth convergence away from at most finitely many singular points. Around these finitely many points one can perform a suitable blow-up and in the limit one finds at most finitely many harmonic maps $\omega_i\in C^\infty (S^2,N)$, $i=1,\ldots,L$, the so called bubbles. Hence one has a clear understanding of the convergence behaviour of such a sequence of harmonic maps away from the finitely many singular points and also close to the singular points via the suitably chosen blow-up. The difficult part is to understand what is going on in the intermediate region which we call the neck region.

Over the last 25 years many people have contributed to a better understanding of the convergence in the neck region and it has been shown that the so called energy identity and the no-neck property hold true, see for example \cite{chen99,ding95,lammsharp,lin98,lin99,lin02,parker96,qing95,qing97,sacks81,yin1,yin2}. Here the energy identity corresponds to the fact that in the limit there is no energy in the neck region and the no-neck property shows that the weak limit $u_0$ and the bubbles $\omega_i$ are actually all pointwise connected.
Thus there is a very satisfying theory available for the convergence of the sequence $u_k$.

In contrast to this, here we are not interested in the behaviour of the energy but of the index of a sequence of harmonic maps with uniformly bounded energy. In particular we are interested in upper and lower bounds on ``the index of the limiting bubble tree''.
Recall that the index $\Ind(u)$ is defined to be the dimension of the maximal subspace on which the second variation of the Dirichlet energy $E$ is negative definite. Recall that the second variation of $E$ is given by 
\[
D^2E(u)(X,X)= \int_{M} (|\nabla X|^2 - \langle A^2_u(X),X \rangle)\, dv_g,
\]
where $X\in W^{1,2}(M,u^*TN):=\{ X\in W^{1,2}(M,\R^m):\, X(x)\in T_{u(x)}N \ \ \text{for}\ \ \text{a.e.} \ \ x\in M\}$ and
\[
\langle A^2_u(X),Y\rangle =\langle A_u(\nabla u,\nabla u), A_u(X,Y)\rangle
\]
for $X,Y\in W^{1,2}(M,u^*TN)$.
Additionally, the nullity $\Nul(u)$ is defined to be the dimension of the kernel of the bilinear form associated to $D^2E(u)$. We remark that harmonic maps with a controlled index have been constructed respectively studied in \cite{micallefmoore,sun1,sun2}. 

To be more precise, the goal of this paper is to show bounds for the index of the sequence $u_k$ in terms of the index (and nullity) of the limiting objects $u_0$ respectively $\omega_i$ with $1\leq i \leq L$.
Our first Theorem is a lower bound and reads as follows (see section \ref{lower bound} for more details).
\begin{theorem}\label{main1}
Let $u_k\in W^{1,2}(M,N)$ be a sequence of harmonic maps with uniformly bounded energy $E(u_k) \leq C$ and with a weak limit $u_0\in C^\infty(M,N)$ and finitely many bubbles $\omega_i\in C^\infty(M,N)$, $1\leq i \leq L$, as described above. Then we have the estimate
\[
\Ind(u_0)+\sum_{i=1}^l \Ind(\omega_i) \leq \liminf_{k\to \infty} \Ind(u_k),
\]
where $\Ind(\cdot)$ denotes the index of the corresponding map.
\end{theorem}
Note that this result relies on a rather standard capacity argument and has already been observed in \cite{stern}. We include it here for the sake of completeness.

The much harder result is the upper bound for the index along the sequence $u_k$. The main result of our paper is the following.
\begin{theorem} \label{main2}
Let $u_k$, $u_0$ and $\omega_i$, $1\leq i \leq L$ be as in Theorem \ref{main1}. Then we have the estimate
\[
\limsup_{k\to \infty} (\Ind(u_k) +\Nul(u_k)) \leq \Ind(u_0)+\Nul(u_0)+\sum_{i=1}^L (\Ind(\omega_i)+\Nul(\omega_i)).
\]
\end{theorem}

Related results also in the context of harmonic maps \cite{yin1,yin2} and critical points of conformally invariant variational problems \cite{dalio} have been obtained recently. See Remark \ref{conflict} for further comments on these works.

We note that a similar result has been obtained in the very influential work of Chodosh and Mantoulidis \cite{chodoshmantoulidis} on the Allen-Cahn approximation of minimal surfaces and by Marques and Neves \cite{marquesneves} in the context of minimal surfaces.

We also remark that in this Theorem we can not get rid of the nullity in general, since there is apriori no reason which excludes that we can have a sequence of eigenvalues $\lambda_k<0$ converging to zero. 

Let us now comment on the main ingredients for the proofs of our Theorems. Note that for both Theorems we have to carefully study the behaviour of the second variation resp. its eigenfunctions on the various regions on $M$ corresponding to the situations where we have strong convergence to $u_0$, strong convergence to the bubbles $\omega_i$ or the neck region. 

In order to show Theorem \ref{main1} we construct for given $X_0\in W^{1,2}(M,u_0^*TN)$ resp. $Z_i\in W^{1,2}(S^2,\omega_i^*TN)$ with $1\leq i \leq L$, a sequence $X_k\in W^{1,2}(M,u_k^*TN)$ such that 
\[
D^2E(u_k)(X_k)\to D^2E(u_0)(X_0)+\sum_{i=1}^L D^2E(\omega_i) (Z_i)
\]
from which the claim then easily follows. 

The argument for Theorem \ref{main2} is much more involved since we have to understand the convergence of eigenfunctions of $D^2E(u_k)$ in great detail. In order to do this we first define a family of bilinear forms associated and varying along the sequence $u_k$. More precisely, we let
\[
\langle \langle X,Y \rangle \rangle_{u_k} = \int_M (\langle X, Y \rangle +  \langle A^2_{u_k}(X),Y \rangle) \, dv_g.
\]
We show that this bilinear form is actually a scalar product once a suitable isometric immersion $N\hookrightarrow \R^m$ has been fixed, which we do from then on. It then follows from the fact that the index and the nullity are independent of the underlying vector space and the scalar product, that we can diagonalize $D^2E(u_k)$ with respect to the scalar product $\langle \langle \cdot,\cdot \rangle \rangle_{u_k}$. In order to study the index and the nullity along the sequence $u_k$ we then study sequences of eigenfunctions $X_k\in W^{1,2}(M,u_k^*TN)$ corresponding to an arbitrary eigenvalue $\lambda_k$, i.e. solutions of the linear PDE
\[
 P(u_k)(-\Delta X_k - A^2_{u_k}(X_k))= \lambda_k (X_k+ A^2_{u_k}(X_k)) ,
\]
where $P(u_k)$ denotes the orthogonal projection of $\R^m$ onto $T_{u_k}N$. Without loss of generality we assume the $X_k$ to be normalized in the sense that $\langle \langle X_k,X_k \rangle \rangle_{u_k}=1$.

The reason why we choose $\langle \langle \cdot,\cdot \rangle \rangle_{u_k}$ as our underlying scalar product has two main reasons. The first one is that by multiplying the equation for $X_k$ with $X_k$ itself and using our normalization one directly obtains a uniform lower bound on the eigenvalues, namely
\[
\lambda_k\geq -1
\]
and a uniform upper bound for the $W^{1,2}$-norm of $X_k$, namely
\[
\|X_k\|_{W^{1,2}(M,u_k^* TN)} \leq \sqrt{2}.
\]
These estimates directly imply good convergence properties of the $X_k$ in the region where $u_k$ strongly converges to $u_0$. The second reason for the choice of the scalar product is the conformal invariance of the second term in its definition. This fact turns out to be crucial when studying the convergence of the $X_k$ in the bubble region.

Finally, we show that the energy of the $X_k$ converges to zero in the neck region. In order to do this we adjust and simplify some of the arguments of \cite{lin98,qing97,sacks81}. To conclude, we are able to show a variant of the energy identity for the sequence $X_k$ which together with the fact that we can show that orthogonality is preserved along the sequence, implies the upper bound which we claim in Theorem \ref{main2}. 
\medskip

In the following we give an outline of our paper.
\medskip

In section $2$ we recall the necessary results about the bubble convergence of a sequence of harmonic maps with uniformly bounded energy.
\medskip

Section $3$ contains the definition of the bilinear form $\langle \langle \cdot,\cdot \rangle \rangle_{u}$ and the proof of the fact that we can choose an isometric immersion such that it is actually a scalar product.
\medskip

In section $4$ we show Theorem \ref{main1}.
\medskip

In the most important section $5$ we prove Theorem \ref{main2} by a careful study of the convergence properties of the sequence $X_k$ on the various regions in $M$. As a byproduct of our analysis on the neck region, we also obtain a shorter proof of the no-neck property for sequences of harmonic maps.
\medskip

In the short section $6$ we sketch how to modify our arguments to handle the case of sequences of critical points of any two-dimensional conformally invariant variational integrand. Thus we recover the main result of \cite{dalio} using quite different techniques.
\medskip

Finally, in the appendix, we recall that the index and nullity are independent of the chosen vector space and scalar product. 

\bigskip
\begin{remark}\label{conflict}
While we were finishing writing our paper we became aware of a Preprint on the arXiv by Francesca Da Lio, Matilde Gianocca and Tristan Rivi\`ere \cite{dalio}, in which the authors obtain a related result as our Theorem \ref{main2} for sequences of critical points of general conformally invariant quadratic variational problems (compare with section $6$ of the present paper). The main difference between the two results is that in \cite{dalio} the authors assume less regularity on the data (such as the target manifold and the differential forms involved). 
The two results were obtained independently with a rather different proof. We believe that our argument, especially because of the choice of the globally defined bilinear form $\langle \langle \cdot,\cdot \rangle \rangle_{u}$, will be of independent interest. We want to highlight again that it is because of this choice that we get uniform lower bounds for the eigenvalues and uniform upper bounds for the $W^{1,2}$-norm of the eigenfunctions basically for free.
\medskip

Only recently, we also became aware of two papers of Hao Yin \cite{yin1,yin2} who also proved Theorem \ref{main2} for sequences of harmonic maps. His argument relies on a rather delicate analysis of the neck region and is also of independent interest. His analysis was focused on the behavior of harmonic maps in the neck region and in contrast, we provide a self-contained analysis on all three regions arising via the bubbling process on $M$.
\end{remark}

\section{Review of the bubbling analysis for sequences of harmonic maps}
In this section we briefly recall the general bubbling behaviour of sequences of harmonic maps $u_k\in W^{1,2}(M,N)$, $k\in \N$, between a Riemann surface $M$ and a Riemannian manifold $N$ with uniformly bounded Dirichlet energy $E(u_k)\leq C$. Due to the regularity result of H\'elein (see e.g. \cite{helein02}) we can and will assume that $u\in C^\infty(M,N)$. We would like to emphasise that ``being harmonic'' does not depend on the isometric embedding chosen for $N$, see e.g. \cite[Lemma 4.1.2]{helein02}.

The study of sequences of harmonic maps $(u_k)\in W^{1,2}(M,N)$ with uniformly bounded energy was initiated in the paper of Sacks and Uhlenbeck \cite{sacks81} in which they showed that there exists a weak limit $u_0\in W^{1,2}(M,N)$ such that $u_k$ also converges locally strongly (always after the selection of appropriate subsequences) away from an at most finite set $\Sigma_0 \subset M$. The elements in $\Sigma_0$ are characterized by the fact that the Dirichlet energy locally concentrates around these points. After showing this result, Sacks and Uhlenbeck were also able to perform a blow-up around the finitely many points in $\Sigma_0$ and in the limit they found non-trivial harmonic maps $\omega_i:S^2\to N$, the so called bubbles. Later on this argument was extended by Parker \cite{parker96} and he was able to find all possible energy concentration points and all possible bubbles (a modified version of this result in a more general setting, including the case of sequences of critical points of conformally invariant variational problems, can be found in the paper of Sharp and the second author of this paper \cite{lammsharp}).  In summary, the authors showed the existence of a weak limit $u_0:M\to N$, an at most finite set of energy concentration points $\Sigma_0$ of the sequence $u_k$, at most finitely many bubbles $\omega_i:S^2\to N$ with $1\leq i \leq L$ and at most finite sets $\Sigma_i\subset \C$ of energy concentration points on the bubbles $\omega_i$. Note that here and in the following we freely identify the base $S^2$ of a bubble $\omega_i$ with $\C$.
Additionally, for every $1\leq i \leq L$, there exist conformal dilations $m_k^i(z)=p_k^i + r_k^i z$ with $p_k^i\to p\in \Sigma_0$ and $r_k^i\to 0$ so that   
\begin{align*}
	u_k \to u_0 &\text{ in } C^2_{loc}(M\setminus \Sigma_0, N)\\
	u_k\circ m_k^i \to \omega_i &\text{ in } C^2_{loc}(\C\setminus \Sigma_i, N)
 \end{align*}
 It was also shown that the relation
 \begin{align*}
 \max \{  \frac{r_k^{i}}{r_k^{j}},\frac{r_k^{j}}{r_k^{i}},\frac{\dist(p_k^{i},p_k^{j})}{r_k^{i}+r_k^{j}} \} &\to \infty, \ \ \ \forall \ \ 1\leq i,j \leq L, \ \  i\not= j, 
\end{align*}
as $k\to \infty$, holds true.

Moreover, there exists a conformal parametrization $n_k^{i}(t+i\theta) = p_k^i + \varrho_k e^{-(t+i\theta)}$ for $[-L_k^i,L_k^i]\times T^1$ (here $T^1=S^1$ with the flat metric) with $L_k^i=\frac12 \log(\frac{s_k^i}{r_k^i})$ and $\varrho^i_k=\sqrt{s_k^i r_k^i}$ (where $s^i_k\to 0$ is a sequence so that $\frac{s_k^i}{r_k^i}\to \infty$ as $k\to \infty$), related to the neck that attaches the bubble $\omega_i$ either to the base $u_0$ or to another bubble $\omega_j$, so that 
\begin{align*}
 u_k\circ n_k^{ij} \to 0 &\text{ in } W^{1,2}_{loc} (\R\times T^1).
\end{align*}
This last fact is known as the so called energy identity. Note that it is also true that the oscillation of $u_k\circ n_k^{ij}$ converges to zero and this is known as the no-neck property. We provide an alternative and short proof of this fact in section \ref{section5} as a consequence of our general neck analysis.

In the following we will use these convergence properties and our notations. We will also say that the sequence $u_k$ bubble-converges to $u_0 \cup_{i=1}^L \omega_i$.

\section{Definition of the scalar product and the index}
In this section we define scalar products which are well adjusted to the study of the index and the nullity of a harmonic map $u:M\to N$. We start with the following bilinear forms.
\begin{definition}
For a map $u\in C^1(M,N)$ we define two bilinear forms on $u^*TN$ by
\begin{align*}
\langle \langle X,Y \rangle \rangle_u &:= \int_{M} (\langle X, Y \rangle +  \langle A^2_u(X),Y \rangle) \, dv_g, \ \ \ \text{resp.}\\
\langle X,Y  \rangle_u &:= \int_{M} \langle A^2_u(X),Y \rangle \, dv_g,
\end{align*}
where $X,Y\in C^0(M,u^*TN)$ and $\langle A^2_u(X),Y\rangle =\langle A_u(\nabla u,\nabla u), A_u(X,Y)\rangle$ is the term which arises in the second variation of the Dirichlet energy.
\end{definition}
One key advantage of the second scalar product is it's invariance under conformal transformations of $M$. More precisely, if $m:M\to M$ is a conformal transformation, then 
\[
A^2_{u\circ m}(X\circ m) =\frac12|\nabla m|^2 (A^2_u(X))\circ m.
\]
Thus, this part of the bilinear form is well-adapted to the later study of the index for sequences of harmonic maps on the domains where the maps form bubbles.

In the next Lemma we show that we can choose an isometric immersion $N\hookrightarrow \R^m$ so that the two bilinear forms defined above actually become scalar products. In the following we will always work with this specific isometric immersion.
\begin{lemma}\label{lem.choiceoftheembedding}
Let $(N,h)$ be a smooth and closed Riemannian manifold. 
Then there exists an isometric immersion $i:N\to \R^n$, so that for some constant $0<c<\infty$ we have
\[
\langle A(v,v),A(w,w) \rangle \geq c|v|^2|w|^2
\]
for every $p\in N$ and every $v,w\in T_pN$.
\end{lemma}
\begin{proof}
Let $j:N\to \R^n$ be an isometric immersion as for example given by the celebrated result of Nash. For $\lambda>1$ we consider the map $F_\lambda:\R^n \to \R^{2n}$ given by
\[
F_\lambda(x)= \left( \frac{\cos(\lambda x^1)}{\lambda} ,\frac{\sin(\lambda x^1)}{\lambda},\dots, \frac{cos(\lambda x^n)}{\lambda}, \frac{\sin(\lambda x^n)}{\lambda} \right).
\]
Note that for every $\lambda \geq 1$ we have
\[
F_\lambda^*\delta_{\R^{2n}} =\delta_{\R^n}.
\]
Next we let $f_\lambda:=F_\lambda \circ j:N\to \R^{2m}$ and we note that the second fundamental form of this map is given by
\begin{align*}
 A_{f_\lambda} (X,Y)=A_{F_\lambda}(j_\sharp X,j_\sharp Y) +(F_\lambda)_\sharp A_j(X,Y).   
\end{align*}
This can be seen by taking a geodesic arc $\gamma:(-\varepsilon,\varepsilon)\to N$ which is parametrized by arc-length, and noting that
\begin{align*}
(f_\lambda\circ \gamma)''(t) &=(df_\lambda(\gamma(t)) \cdot \gamma'(t))'\\
&=D^2F_\lambda(j(\gamma(t)))(j_\sharp \gamma'(t),j_\sharp \gamma'(t))+DF_\lambda (j(\gamma(t))) (j(\gamma(t)))''\\
&= A_{F_\lambda}(j_\sharp \gamma'(t),j_\sharp \gamma'(t))+(F_\lambda)_\sharp (A_j(\gamma'(t),\gamma'(t)).
\end{align*}
Since $A_{F_\lambda}$ and $(F_\lambda)_\sharp A_j$ are orthogonal and
\[
A_{F_\lambda} =-\lambda F_\lambda Id
\]
we obtain
\begin{align*}
\langle A_{f_\lambda(x)} (v,v), A_{f_\lambda(x)} (w,w) \rangle =    \lambda^2  |v|^2|w|^2 +\langle A_j(v,v),A_j(w,w)\rangle.
\end{align*}
Hence we can choose $\lambda$ sufficiently large so that the claim holds for the fixed isometric immersion $i:=\frac{1}{\sqrt{2}} (id_{\R^n},f_\lambda):N\to \R^{3n}$.
\end{proof}
The bilinear form associated to the second variation of the energy that determines the index of a harmonic map is given by
\begin{equation}\label{eq.bilinear form}
    D^2E(u)(X,X)= \int_{M} (|\nabla X|^2 - \langle A^2_u(X),X) \rangle\, dv_g\,.
\end{equation}
Due to the general index considerations in appendix \ref{sec.appA}, i.e. Sylvester's law of inertia, we can choose the scalar product freely. For our purposes we choose $\langle \langle \cdot, \cdot \rangle\rangle_u$. As an immediate consequence of appendix \ref{sec.appA} we have the Lemma

\begin{lemma}\label{lem.EquationEigenvectors}
Let $u:M \to N$ be a harmonic map. Then the following two statements are equivalent:
\begin{enumerate}
    \item[(i)] $X$ is an eigenvector of $D^2E(u)$ to the eigenvalue $\lambda$;
    \item[(ii)] $X$ solves 
   \begin{equation}\label{eq.EquationEigenvector1}
    P(u)(-\Delta X - A^2_u(X))= \lambda (X+ A^2_u(X)) \,.
\end{equation}
\end{enumerate}
Moreover, the equation \eqref{eq.EquationEigenvector1} can be rewritten in the form
\begin{equation}\label{eq.EquationEigenvector2}
   -\Delta X + B_u^i \partial_i X + C_uX = \lambda X \,,
\end{equation}
where the coefficients $B_u^i$ and $C_u$ are determined by geometric quantities of $N$ and evaluated at $u$. In particular, they satisfy the bounds
\[|B_u^i|\lesssim |\nabla u| \text{ and } |C_u|\lesssim (1+|\lambda|) |\nabla u|^2\,.
\]
\end{lemma}
\begin{proof}
Note that by the above discussion we only have to show that $(ii)$ implies \eqref{eq.EquationEigenvector2}. For this we note that
\begin{align*}
    \Delta X &= \Delta (P(u)X) = P(u)\Delta X + 2 \nabla P(u)\,\nabla X +(\Delta P(u))X \\&= P(u)\Delta X + 2 dP(u)\nabla u\,\nabla X+ \left(dP(u)(A(u)(\nabla u, \nabla u) + D^2P(u)(\nabla u, \nabla u) \right) X ,
\end{align*}
where we used that $X\in u^* TN$ and that $\Delta u = (\Delta u)^\perp$. Thus the claim follows.
\end{proof}
Based on the considerations in the appendix, we can use the following considerations for the definition of the index of a harmonic map $u\colon M \to N$.
\begin{definition}
Let $u:M\to N$ be harmonic. Then the index is defined by
\begin{align*}
\Ind(u) =\dim \{ X\in H^1(u^* TN): \, &P(u)(-\Delta X-A^2_u(X))=\lambda(X+ A^2_u(X))\ \ \ \text{s.t.} \\ &\lambda <0\}
\end{align*}
and the nullity by
\[
\Nul(u) =\dim \{ X\in H^1(u^* TN):\, P(u)(-\Delta X-A^2_u(X))= 0\}.
\]
\end{definition}
It again follows from the considerations in the Appendix that these objects are well-defined.

\section{Index lower bound}\label{lower bound}

In this section we prove Theorem \ref{main1} and we start 
with a technical Lemma on the density of certain vectorfields. 
\begin{lemma}
Let $u:M\to N$ be harmonic and let $P=\{p_1,\ldots,p_l\} \subset M$ be a finite set of points. Then $C_P(M,u^* TN):=\{ X\in C^\infty (M,u^* TN):\, \spt(X) \subset \subset M\backslash P\}$ is dense in $H^1\cap L^\infty(M,u^* TN)$. This implies that for every $X\in H^1\cap L^\infty (M,u^* TN)$ there exists a sequence $X_k\in C_P(M,u^* TN)$ so that 
\[
\lim_{k\to \infty} D^2E(u)(X_k) =D^2E(u)(X).
\]
\end{lemma}
\begin{proof}
This follows from the fact that a finite set of points has capacity zero and the density of $C^\infty(M,u^*TN)$ in $(H^1\cap L^\infty)(M,u^* TN)$. Alternatively, one can use a logarithmic cut-off function in order obtain the same result.
\end{proof}
The next result is the crucial convergence statement for the second variation of the Dirichlet energy and we note that Theorem \ref{main1} is a direct consequence of the following Theorem.
\begin{theorem}
Let $u_k\in C^\infty(M,N)$ be a sequence of harmonic maps which bubble converges to $u_0\cup_{i=1}^L \omega_i$. Then for every $X_0\in C^\infty (M,u_0^*TN)$ and $Z_i\in C^\infty(S^2,\omega_i^*TN)$ with $1\leq i \leq L$, there exists a sequence $X_k\in C^\infty (M,u_k^*TN)$ so that 
\begin{align*}
\lim_{k\to \infty} D^2E(u_k)(X_k)= D^2E(u_0)(X_0)+\sum_{i=1}^l D^2E(\omega_i)(Z_i).
\end{align*}
\end{theorem}
\begin{proof}
Due to the previous Lemma we may assume that $X_0$ resp. $Z_i$, $1\leq i \leq L$, vanishes in a neighbourhood of every energy concentration point of the bubble tree.

Since the maps $u_k$ bubble converge to $u_0 \cup_{i=1}^L \omega_i$ by our assumption, there exist associated conformal rescalings $m^i_{k}$ so that
\[
u_k \circ m_{i,k} \to \omega_i
\]
in $C^2_{loc}(\R^2 \backslash \Sigma_i)$.

Furthermore, the conformal invariance of $E$ implies that 
\[
D^2 E(u) (X)=D^2 E(u\circ m) (X\circ m)
\]
for every map $u$ and every conformal rescaling $m$. 

Hence, we have with $Z_{i,k}=(P(u_k)Z_i)\circ (m_k^i)^{-1}\in C^\infty(M,u_k^*TN)$ (note that this is well-defined because $Z_i$ vanishes in a neighbourhood of every energy concentration point)
\[
\lim_{k\to \infty} D^2E(u_k) (Z_{i,k})=\lim_{k\to \infty} D^2E(u_k\circ m_{i,k})(P(u_k)Z_i)=D^2E(\omega_i)(Z_i).
\]
Therefore, the sequence 
\[
X_k=P(u_k)X_0 +\sum_{i=1}^L Z_{i,k}
\]
has all the desired properties. For this we note that for $k$ sufficiently large, the vectorfields $Z_{i,k}$ have pairwise disjoint support.
\end{proof}

\section{Index upper bound}\label{section5}
The index upper bound will be a consequence of the following stronger theorem. 
\begin{theorem}\label{thm.convergenceeigenfunctions}
	Let $u_k\colon M \to N$ be a sequence of harmonic maps, which bubble converges to $u_0 \cup \{ \omega_i \}$ where $u_0\colon M \to N$ is the weak limit and the $\omega_i\colon S^2\to N$ are the bubbles.\\
	Let $X_k \in H^1(M, u^* TN)$ be a sequence of  eigenfunctions of $D^2E(u_k)$ with eigenvalues $\lambda_k \to \lambda$. Passing to a subsequence we find 
	\begin{enumerate}
		\item an eigenfunction $Y_0$ of $D^2E(u_0)$ with eigenvalue $\lambda$ and
		\item eigenfunctions $Z_i$ of $D^2E(\omega_i)$ with eigenvalues $\lambda$ 
	\end{enumerate}
	such that 
	\[\lim_{k\to \infty} \langle\langle X_k,X_k\rangle \rangle_{u_k} = \langle \langle Y_0,Y_0\rangle \rangle_{u_0} + \sum_{i} \langle Z_i,Z_i\rangle_{\omega_i}\,. \]
\end{theorem}

We split the proof of the above Theorem into various subsections devoted to the different regimes of convergence: the base, the bubbles and the neck. Each subsections describes the convergence of the eigenfunctions in an independent proposition. In the final subsection we put them together to conclude the above theorem. 

Before we come the finer analysis let us note the following a priori bound: 
\begin{proposition}\label{energybound}
Let $X$ be a solution of \eqref{eq.EquationEigenvector1} with $\langle \langle X,X\rangle \rangle_u =1$, then we have
\[
\lambda \geq -1
\]
and
\[
\int_M(|X|^2+ |\nabla X|^2) \leq 1+|\lambda|.
\]
In particular, for every eigenvectorfield $X$ with eigenvalue $\lambda \leq 0$ we get the uniform bound
\[
\int_M(|X|^2+ |\nabla X|^2) \leq 2.
\]
\end{proposition}
\begin{proof}
Testing the equation \eqref{eq.EquationEigenvector1} with X implies by our normalization $\langle \langle X,X\rangle \rangle_u =1$ that
\[
\int_M (|X|^2+|\nabla X|^2)= (1+\lambda) \int_M  (|X|^2+ \langle A^2_u(X),X \rangle)  =1+\lambda
\]
and this implies all the claims.
\end{proof}

\subsection{Preliminary ``neck'' analysis}\label{subsec.neck analysis}
In this subsection we provide an a priori estimate on the tangential derivative to the solution of the Poisson's equation on a long neck $[-L, L]\times T^1$. As a byproduct of our analysis we obtain a quantitative version of the no neck property for a sequence $u_k$.
Our proof follows along the lines of the original paper of Sacks-Uhlenbeck, compare \cite[theorem 3.6]{sacks81}.

\begin{lemma}\label{lem.tangential estimate}
	Let $L\in \N$ and let $\phi \in W^{1,2}([-L, L]\times T^1)$ be a solution of 
	\[- \Delta \phi=f. \]
	Then we have the tangential estimate for all $|t_0|\le L-1$
 \begin{align}\label{eq.tangentialestimate}
     \int_{\{t_0\}\times T^1} |\partial_\theta \phi|^2& + \int_{[t_0-1,t_0+1]\times T^1} |\nabla \partial_\theta \phi|^2 \lesssim e^{-\frac19 (L-|t_0|)} \int_{[-L,L]\times T^1} |\nabla \phi|^2 \nonumber\\&+ \int_{[-L,L]\times T^1} \min\{e^{\frac19(1-|t-t_0|)},1\} \, |f|^2=:I\,,
 \end{align}
 where here and in the following we use the notation $a\lesssim b$ iff there exists a uniform constant $C>0$ such that $a\leq Cb$.
 
 Furthermore we have the following $L^\infty$-bound for all $p_0=t_0+i\theta_0$ with $|t_0|<L-1$ and $\theta_0\in [0,2\pi]$
 \begin{align}\label{eq.Linftybound}
     |\phi(p_0)|\le \max\{ \mint_{\{-L\}\times T^1} \phi,\mint_{\{L\}\times T^1} \phi \} + \int_{[-L,L]\times T^1} (L-|t|) |f| + c I^\frac12
 \end{align}
\end{lemma}
\begin{proof}
We start with the tangential integral estimate \eqref{eq.tangentialestimate}.

	For this we define the piecewise linear function 
$q\colon [-L,L] \to \R^n$ as the linear interpolation between $\fint_{\{k\}\times T^1} \phi$ and $\fint_{\{k+1\}\times T^1} \phi$ for each $k \in \{-L, -L+1, \dotsc, L-1\}$. Furthermore we define the $W^{1,2}$-function $\psi=\phi-q$. Note that for each $k$ we have 
\[- \Delta \psi = f \text{ on } [k,k+1]\times T^1\text{ and } \mint_{\{t\} \times T^1} \psi =0 \text{ for } t \in \{k,k+1\}\,.\]
Since $q$ is harmonic we estimate 
\begin{align}\label{eq.simpleobservation}
	\int_{[k,k+1]\times T^1} |\nabla \psi|^2 &= \int_{[k,k+1]\times T^1} |\nabla \phi|^2- |\nabla q|^2 - 2 \int_{\partial [k,k+1]\times T^1} \frac{\partial q}{\partial \nu} (\phi-q)\nonumber \\
	&= \int_{[k,k+1]\times T^1} |\nabla \phi|^2- |\nabla q|^2\le \int_{[k,k+1]\times T^1} |\nabla \phi|^2
\end{align}

Testing the equation for $\psi$ on a fixed interval $[a,b]$ of type $[k, k+\delta]$ or $[k+\delta, k+1]$ for some $0\le \delta \le 1$ with $\psi$ itself gives 
\begin{align}
\int_{[a,b]\times T^1} |\nabla \psi|^2 =&  \int_{\partial [a,b]\times T^1} \frac{\partial \psi}{\partial \nu} \psi +\int_{[a,b]\times T^1} \psi f \nonumber \\
\leq& \int_{\partial [a,b]\times T^1} \frac{\partial \psi}{\partial \nu} \psi +c \int_{[a,b]\times T^1} |f|^2 +\frac13 \int_{[a,b]\times T^1} |\nabla \psi|^2, \label{eq.oneinterval}
\end{align}
where we have used Poincar\'e inequality 
\[\int_{[a,b]\times T^1} |\psi|^2 \le C \int_{[a,b]\times T^1} |\nabla \psi|^2\, \] relying on the fact that $\fint_{\{t\}\times T^1} \psi =0$ on one of the boundary components. 

Furthermore we claim that 
\begin{equation}\label{eq.boundaryestimate}
	|\int_{\{k+\delta\}\times T^1} \psi \partial_t \psi  | \le 3 \int_{\{k+\delta\}\times T^1} |\nabla \psi|^2 + \frac13 \int_{[a,b]\times T^1} |\nabla \psi|^2\,.
\end{equation}

This can be seen as follows for the case of $[k,k+\delta]$: The fundamental theorem of calculus yields 
\begin{align*}
	|\mint_{\{k+\delta\}\times T^1} \psi |&= | \mint_{\{k+\delta\}\times T^1} \psi- \mint_{\{k\}\times T^1} \psi| =| \frac{1}{2\pi}\int_{[k,k+\delta]\times T^1} \partial_t\psi | \\
	&\le \frac{1}{\sqrt{2\pi}} \left(\int_{[k,k+\delta]\times T^1} |\nabla \psi|^2\right)^\frac12\,.
\end{align*}
Hence we can estimate 
\begin{align*}
	\left|\int_{\{k+\delta\}\times T^1} \psi \partial_t \psi  \right| \le &\left(\int_{\{k+\delta\}\times T^1} |\partial_t\psi|^2\right)^\frac12\left(\int_{\{k+\delta\}\times T^1} |\psi-\mint_{\{k+\delta\}\times T^1} \psi |^2\right)^\frac12\\&+ \left| \mint_{\{k+\delta\}\times T^1} \psi \int_{\{k+\delta\}\times T^1} \partial_t \psi \right| \\\le &\int_{\{k+\delta\}\times T^1} |\nabla \psi|^2 \\
 &+ \left(\int_{\{k+\delta\}\times T^1} |\nabla \psi|^2
	\right)^\frac12   \left(\int_{[k,k+\delta]\times T^1} |\nabla \psi|^2\right)^\frac12\\
 \leq& 3 \int_{\{k+\delta\}\times T^1} |\nabla \psi|^2+\frac13\int_{[k,k+\delta]\times T^1} |\nabla \psi|^2\,. 
\end{align*}
Finally, we note that \[\int_{ \{k\}\times T^1} \psi \partial_t \psi  = \int_{\{k\}\times T^1} \partial_t \psi (\phi- q(k))= \int_{\{k\}\times T^1} \partial_t \phi (\phi-q(k)) \]
only depends on $\phi$ and not on the linear interpolation $q$.
Given any $0\le t\le L$ we may thus sum up the estimate \eqref{eq.oneinterval} on all intervals $[k,k+1]$ for $k \in \{-k_1+1,\dots, k_1\}$ with $k_1=\lfloor t\rfloor$ and the end intervals $[-t,-k_1+1], [k_1,t]$ giving in combination with \eqref{eq.boundaryestimate} the estimate
\begin{align}\label{eq.manyintervals}
	\frac23 \int_{[-t,t]\times T^1} |\nabla \psi|^2 &\le \int_{\partial [-t,t]\times T^1} \frac{\partial \psi}{\partial \nu} \psi + c \int_{[-t,t]\times T^1} |f|^2\nonumber\\
	&\le 3\int_{\partial[-t,t]\times T^1} |\nabla \psi|^2 + c \int_{[-t,t]\times T^1}|f|^2 + \frac13 \int_{[-t,t]\times T^1} |\nabla \psi|^2\,.
\end{align}
But this corresponds to a differential inequality and hence, for any $0\le a\le b \le L$ with $b\in \N$, $b\geq 2$, we have 
\begin{align}\label{eq.tangentialestimate01}
	\int_{[-a,a]}|\nabla \psi|^2 &\le e^{\frac19(a-b)} \int_{[-b,b]\times T^1} |\nabla \psi|^2 + c \int_{a}^b \int_{[-s,s]\times T^1} e^{\frac19(a-s)} |f|^2 ds \nonumber\\
	& \le e^{\frac19(a-b)} \int_{[-b,b]\times T^1} |\nabla \phi|^2 + c \int_{[-b,b]\times T^1} \min\{e^{\frac19(a-|t|)},1\} |f|^2 	
\end{align}
where we use $t$ as a coordinate on $[-L,L]$ and we also apply \eqref{eq.simpleobservation} on each subinterval. 

Finally, we may use classical $L^2$-theory with the trace estimates to deduce that 
\begin{align}\label{eq.L^infty+trace}
	\sup_{t \in [-1,1]} &\int_{\{t\}\times T^1} |\partial_\theta \phi|^2 + \int_{[-1,1]\times T^1} |\nabla \partial_\theta\phi|^2 \lesssim \int_{[-2,2]\times T^1} (|f|^2 + |\partial_\theta \phi|^2 )\nonumber\\
	&\lesssim e^{-\frac19b}\int_{[-b,b]\times T^1} |\nabla \phi|^2+ \int_{[-b,b]\times T^1} \min\{e^{\frac19(1-|t|)},1\} |f|^2
 \end{align}
 Applying the above to the translated functions $\phi_{t_0}(t,\theta)=\phi(t_0+t,\theta)$ and choosing $b=L-|t_0|$ shows the claimed estimate \eqref{eq.tangentialestimate}.

 For the $L^\infty$-bound we observe that the function  
 \[\bar{\phi}(t)= \int_{\{t\}\times T^1} \phi \]
solves the ODE
\[
\bar{\phi}''(t)=\bar{f}(t)
\]
where $\bar{f}(t)=-\int_{\{t\} \times T^1} f$. Next, we let
\[
P(t)=\int_0^t (t-s) \int_{\{s\}\times T^1} |f| \, ds
\]
and we note that 
\[
(P(t)\pm \bar{\phi}(t))''=\int_{\{t\} \times T^1} |f| \mp \int_{\{t\} \times T^1} f\geq 0.
\]
Hence it follows from the maximum principle that
\begin{align*}
\max_{t\in [-L,L]} (P(t)\pm \bar{\phi}(t))=& \max_{t\in \{-L,L\}} (P(t)\pm \bar{\phi}(t))\\
\leq& \max \{\int_{[0,L]\times T^1} (L-|t|)|f|, \int_{[-L,0]\times T^1} (L-|t|)|f|\}\\
&+\max \{ \bar{\phi}(L), \bar{\phi}(-L)\}\,.
\end{align*}
Finally, we note that 
\[
|\phi(t_0)|\leq \mint_{\{t_0\} \times T^1} |\phi(t_0,\theta)| \, d\theta+ \int_{\{t_0\} \times T^1}|\partial_\theta \phi(t_0,\theta)|\, d\theta
\]
and the pointwise bound is thus a consequence of the bound for $\bar{\phi}$ and the integral estimate for the tangential derivative. 
\end{proof}

As a corollary of the differential inequality derived above we obtain a quantitative version of the no-neck property for a sequence of harmonic maps. Related results on cylinders of large but finite length have been obtained previously in \cite{chenliwang}.
\begin{corollary}\label{cor.quantitative no neck}
	There exists a constant $\varepsilon_0>0$ such that for every harmonic map $v\colon (-\infty,L)\times T^1 \to N$ with $L\in \R$ satisfying 
	\begin{itemize}
		\item[(a1)] $\int_{(-\infty, L)\times T^1}|\nabla v|^2 < \infty $;
		\item[(a2)] $\int_{B_1(p)} |\nabla v|^2 \le \varepsilon_0$ for all $p \in [-L,L]\times T^1$.
	\end{itemize}
	then for each $j\in \N$ and $|t|<L-2$ we have
	\begin{equation}\label{eq.neckgradientbound}
		|\nabla^j v|^2(t,\theta) \lesssim e^{\frac{1}{10}(|t|-L)} \int_{[-L,L]\times T^1} |\nabla v|^2\,.
	\end{equation} 
\end{corollary}

\begin{proof}
	We apply the results of the above lemma to the harmonic map equation 
	\[\Delta v = A(v)(\nabla v, \nabla v) \text{ i.e. } f=A(v)(\nabla v, \nabla v)\,.\]
	Firstly, recall that the $\varepsilon$-regularity for harmonic maps implies that for every $\varepsilon_0$ small enough and every $j\in \N$
	\begin{equation}\label{eq.gradient estimate}
		|\nabla^j v|^2(p) \lesssim \int_{B_1(p)} |\nabla v|^2 \lesssim \varepsilon_0 \text{ for all } p \in [-L+1, L-1] \times T^1\,.
	\end{equation}
	Secondly, we recall that the quadratic differential $(\partial_zv)^2 dz \otimes dz$ is holomorphic. Hence the finite energy of $v$ implies that 
	\begin{equation}\label{eq.holomorphichopf}
		\int_{\{t\}\times T^1} |\partial_tv|^2 = \int_{\{t\}\times T^1} |\partial_\theta v|^2 \text{ for all } t<L\,.
	\end{equation}
	This can be seen as follows: Let $t_k \to -\infty$ s.t. $\lim_{k\to \infty} \int_{\{t_k\}\times T^1} |\nabla v|^2=0$. Then we have 
	\begin{align*}
		\int_{\{t\}\times T^1} (\partial_zv)^2 \,d\theta &= \lim_{k} \left(\int_{\{t\}\times T^1} (\partial_zv)^2 \,d\theta - \int_{\{t_k\}\times T^1} (\partial_zv)^2 \,d\theta\right)\\& = \lim_{k} \int_{[t_k,t]\times T^1} \partial_{\bar{z}} (\partial_zv)^2 \,d\bar{z}\wedge dz=0\,.
	\end{align*}
	Taking the real part of this identity implies \eqref{eq.holomorphichopf}.
	Now we can refine the differential inequality \eqref{eq.manyintervals} by noting that due to \eqref{eq.gradient estimate} and \eqref{eq.holomorphichopf} we have for any $0<t<L-1$ 
	\begin{align*}
		\int_{[-t,t]\times T^1} |f|^2 &\le C \sup_{[-t,t]\times T^1} |\nabla v|^2 \int_{[-t,t]\times T^1}|\nabla v|^2 \le C\varepsilon_0 \int_{[-t,t]\times T^1} |\partial_\theta v|^2 .
	\end{align*} 
	Hence, for $\varepsilon_0$ sufficiently small, the differential inequality from the previous lemma reads as follows
	\[ \frac{3}{10}\int_{[-t,t]\times T^1} |\nabla v|^2 \le 3\int_{\partial[-t,t]\times T^1}|\nabla v|^2\,. \]
	We can then proceed as above to deduce that for any $0<a<b\le L-1$ with $b \in \N$ we obtain 
	\begin{align*}
		\frac12 \int_{[-a,a]\times T^1} |\nabla v|^2 =& \int_{[-a,a]\times T^1} |\partial_\theta v|^2 
  \le e^{\frac{1}{10}(a-b)} \int_{[-b,b]\times T^1} |\nabla v|^2\,.
	\end{align*}
	Appealing once more to \eqref{eq.gradient estimate} and applying it to translates of $v$ provides \eqref{eq.neckgradientbound}.
\end{proof}

\subsection{Convergence on the base}\label{subsec.convergence on the base}
From here on we assume that all assumptions of Theorem \ref{thm.convergenceeigenfunctions} are satisfied.
\begin{proposition}\label{prop.base}
	Up to passing to a subsequence there exists $Y_0 \in C^{2,\alpha}(M, u_0^*TN)$ satisfying 
	\begin{align}
		\label{convergence base} X_k \to&  Y_0 \ \ \text{in} \ \  L^2(M)\cap C^{2,\alpha}_{loc}(M\setminus \Sigma_0) \ \ \ \text{and} \\
		\label{equation base} P(u_0)(-\Delta Y_0 -A_{u_0}^2(Y_0)) =& \lambda (Y_0+A_{u_0}^2(Y_0)).
	\end{align}
	In particular we have 
	\begin{align}\label{eq.quantitive convergence base} 
		\lim_{r\to 0} \limsup_{k\to \infty} &\int_{M \cap \{\dist(\cdot,\Sigma_0)>r\}}(|\nabla X_k|^2 + |X_k|^2 +\langle A^2_{u_k}(X_k),X_k\rangle)\nonumber \\= &\int_M (|\nabla Y_0|^2 + |Y_0|^2 +\langle A^2_{u_0}(Y_0),Y_0\rangle)\,.
	\end{align}
\end{proposition}
\begin{proof}
Due to the a priori $W^{1,2}$ bound from Lemma \ref{energybound}, we may pass to a subsequence with $X_k \rightharpoonup Y_0$ in $W^{1,2}(M)$. In particular $X_k \to Y_0$ strongly in $L^2(M)$.
\medskip

Next we claim that the sequence $X_k$ is uniformly bounded in $C^{2}_{loc}(M\setminus \Sigma_0)$.

For $k\in \N$ fixed and $\varphi \in C^\infty_c(M\setminus \Sigma_0)$ it follows from \eqref{eq.EquationEigenvector2} that
\begin{align*}
-\Delta (\varphi X_k)= -B^i_{u_k} \partial_i (\varphi X_k)  -\Delta \varphi X_k -2\nabla \varphi \nabla X_k +B^i_{u_k} X_k \partial_i \varphi +(\lambda_k- C_{u_k}) \varphi X_k\,.
\end{align*}

By our assumptions, Proposition \ref{energybound} and the fact that $u_k \to u_0$ strongly in $W^{1,2}(M\setminus \Sigma_0)$,
it follows that the right hand side of this equation is in $L^2(M)$ and thus $\varphi X_k \in W^{2,2}(M)$. Bootstrapping this argument finally yields the claim since $u_k \to u_0$ in $C_{loc}^{2,\alpha}(M\setminus \Sigma_0)$. This implies that the coefficients in the equations above $B^i_{u_k}, C_{u_k}$ converge in $C^{0,\alpha}(M\setminus \Sigma_0)$ to $B^i_{u_0}, C_{u_0}$ and therefore they are locally uniformly bounded on $M\backslash \Sigma_0$.
\medskip

The claim of the Proposition is now an immediate consequence of the Arzel\`a-Ascoli theorem and the just established uniform a priori bound in $C^{2,\alpha}_{loc}(M\setminus \Sigma_0)$. In particular we deduce the convergence of the equation \eqref{eq.EquationEigenvector2} and 
\[(1-P(u_0(p)))Y_0(p)=\lim_{k\to\infty} (1-P(u_k(p)))X_k(p) = 0 \text{ for every }  p \in M\setminus \Sigma_0\,.\]
Hence $Y_0(p)$ is a tangent vectorfield to $N$ in $u_0(p)$. 
\end{proof}

\subsection{Convergence on the bubble}\label{subsec.convergence on the bubble}
Next, we study the convergence of the sequence $X_k$ on the domains where the bubbles form.
\begin{proposition}\label{prop.bubble}
	For every bubble $\omega_i$, up to passing to a subsequence, there exists $Z_i \in C^2(S^2, \omega_i^*TN)$ satisfying 
	\begin{align}
		\label{convergence bubble} X_k\circ m_k^i \to& Z_i\ \ \ \text{in} \ \ \ C^{2,\alpha}_{loc}(\R^2\backslash \Sigma_i) \ \ \ \text{and}	\\	
 P(\omega_i)(-\Delta Z_i -A_{\omega_i}^2(Z_i)) =& \lambda A_{\omega_i}^2(Z_i).
	\end{align}
	In particular, we have 
	\begin{align}\label{eq.quantitive convergence bubble} 
		\lim_{r\to 0} \limsup_{k\to \infty} &\int_{\C \cap \{\dist(\cdot,\Sigma_i)>r\}\cap B_{\frac1r}}(|\nabla (X_k\circ m_k^i)|^2 +\langle A^2_{u_k\circ m_k^i}(X_k\circ m_k^i),(X_k\circ m_k^i)\rangle) \nonumber \\= &\int_\C(|\nabla Z_i|^2 + \langle A^2_{\omega_i}(Z_i),Z_i\rangle)\,.
	\end{align}
\end{proposition}
\begin{proof}
We start by remarking that for a conformal transformation $m:M  \to M$ the Laplacian and the operator $A_u^2$ transform according to the rule
\[
-\Delta (X\circ m)= \frac12|\nabla m|^2 (-\Delta X)\circ m
\]
resp.
\[
A^2_{u\circ m}(X\circ m) =\frac12|\nabla m|^2 (A^2_u(X))\circ m.
\]
Thus, the eigenfunction equation \eqref{eq.EquationEigenvector1} transforms as follows
\begin{align}
    P(u\circ m) (-\Delta (X\circ m))=(1+\lambda) A^2_{u\circ m} (X\circ m)+\frac\lambda2 |\nabla m|^2 X\circ m. \label{eq.eigenfunctionmoebius} 
\end{align}
If we now suppose that $m^i_k:\R^2\to M$ (here we assume loss of generality that we can localize everything to work on $\R^2$) is a sequence of conformal transformations such that $u_k\circ m_k^i$ converges strongly in $W^{2,2}_{loc}(\R^2 \backslash \Sigma_i,N)$ to a non-trivial bubble $\omega_i:\C \to N$. Then we look at $X_k\circ m^i_k$ and we claim that $X_k\circ m^i_k$  converges in $C^{2,\alpha}_{loc}(\R^2 \backslash \Sigma_i)$, for every $0<\alpha<1$, to a map $Z_i\in H^1(S^2,\omega_i^*TN)$ which is an eigenfunction of $D^2E(\omega_i)$ for the eigenvalue $\lambda$. We note that along the bubbling process our scalar product $\langle \langle \cdot, \cdot \rangle \rangle_{u_k}$ converges to $\langle \cdot, \cdot, \rangle_{u_k}$.

Since $\omega_i$ is non-trivial, we can find a compact subset $K\subset \R^2\setminus \Sigma_i$ and $\delta >0$ such that 
\[
E_i:=\{ x\in K:\, |\nabla \omega_i|(x)>2\delta\}
\]
has positive measure. It follows from the bubble convergence that $u_k \circ m^i_k$ converges to $\omega_i$ in $C^2_{loc}(\R^2\backslash \Sigma_i)$, and hence there exists a number $k_0\in \N$ so that
\[
E_i\subset \{x\in K:\, |\nabla (u_k\circ m^i_k) (x)| >\delta\}
\]
for every $k\geq k_0$. Thus, we obtain
\[
\int_{E_i} |X_k\circ m^i_k|^2\leq \frac{1}{\delta^2}\int_{E_i} |\nabla (u_k\circ m^i_k)|^2 |X_k\circ m^i_k|^2 \leq \frac{c}{\delta^2} \langle \langle X_k,X_k \rangle \rangle_{u_k} \leq \frac{c}{\delta^2},
\]
where we used the conformal invariance of the second term in the definition of $\langle \langle \cdot, \cdot \rangle \rangle_u$ and Lemma \ref{lem.choiceoftheembedding}.

Next we use this estimate in order to bound
\begin{align*}
|\mint_{E_i} X_k\circ m^i_k|^2\leq \mint_{E_i} |X_k\circ m^i_k|^2 \leq \frac{c}{\delta^2|E_i|}.
\end{align*}
Together with the bound
\begin{align*}
\int_K |X_k\circ m^i_k-\mint_{E_i} X_k\circ m^i_k|^2\leq&
c\int_K |X_k\circ m^i_k-\mint_K X_k\circ m^i_k|^2 \\
&+c|K| |\mint_{E_i} X_k\circ m^i_k-\mint_K X_k\circ m^i_k|^2\\ \leq& c\|\nabla X_k \|^2_{L^2(K)} \\
&+\frac{c|K|}{|E_i|}\|X_k \circ m^i_k-\mint_K X_k\circ m^i_k\|^2_{L^2(E_i)}\\ \leq& c(1+\frac{|K|}{|E_i|}),
\end{align*}
where we used again the conformal invariance of $\|\nabla X_k\|_{L^2}$ and Lemma \ref{energybound}.
This implies for every compact subset $K \subset \R^2\backslash \Sigma_i$ with positive measure the uniform $W^{1,2}$-bound
\begin{align*}
\|X_k\circ m^i_k\|_{L^2(K)}+\|\nabla (X_k\circ m^i_k)\|_{L^2(K)} \leq c(1+\frac{1}{\delta \sqrt{|E_i|}}+\sqrt{\frac{|K|}{|E_i|}}).
\end{align*}
Thus we can argue as in the proof of Proposition \ref{prop.base} in order to use standard PDE techniques to get locally uniform $C^{2,\alpha}$-bounds, for every $0<\alpha<1$, and this finishes the above claim about the convergence of the eigenfunctions on the bubbles once we observe that in our case $|\nabla m^i_k|\to 0$ as $k\to \infty$. In particular this shows that the term $\lambda_k |\nabla m^i_k|^2 X_k\circ m^i_k$ in the PDE \eqref{eq.eigenfunctionmoebius} disappears in the limit and there the claimed convergence in \eqref{eq.quantitive convergence bubble} holds. Note that the possible point singularities of the limit $Z_i$ can be removed by standard capacity considerations.
It then follows from the same arguments as in the proof of Proposition \ref{prop.base} that $Z_i \in C^2(S^2,\omega_i^*TN)$.
\end{proof}

\subsection{Convergence on the neck}\label{subsec.convergence on the neck}
Finally, we study the convergence on the neck regions.
\begin{proposition}\label{prop.neck}
	For every neck we have $X_k\circ n_k^{i} \to 0$ in $C^{2,\alpha}_{loc}(\R \times T^1)$
\end{proposition}
In particular we have 
\begin{align}\label{eq.quantitive convergence neck} 
		\lim_{l\to \infty} \limsup_{k\to \infty} &\int_{[-L_k^i+l,L_k^i-l]\times T^1}|\nabla (X_k\circ n_k^{i})|^2 +\langle A^2_{u_k\circ n_k^{i}}(X_k\circ n_k^{i}),X_k\circ n_k^{i}\rangle =0\,.
	\end{align}
\begin{proof}
In a first step we show that the sequence $X_k\circ n_k^{i}$ is uniformly bounded in $L^\infty$.
\medskip

For a fixed $k$ we consider the vectorfield $\tilde{X}_k^i=X_k\circ n_k^{i}$ related to the harmonic map $v_k^i=u_k\circ n_k^{i}$ solving the conformal transformed equation \eqref{eq.eigenfunctionmoebius} with $m_k^i=n_k^{i}$ in particular it is a solution of type, compare \eqref{eq.EquationEigenvector2}, 
\[-\Delta \tilde{X_k^i} =- B^j_{v_k^i} \partial_j \tilde{X}_k^i - C_{v_k^i} \tilde{X}_k^i +\frac{\lambda_k}{2}|\nabla n_k^i|^2 \tilde{X}_k^i=:f_k^i\,.\]
In the rest of the proof we drop the indices $k$ and $i$ since they are fixed.

We note that 
\[|f|\lesssim  |\nabla v|\left(|\nabla \tilde{X}| + |\nabla v| |\tilde{X}|\right) + |\nabla n|^2 |\tilde{X}|\,\text{ and } |n|,\frac{1}{\sqrt{2}}|\nabla n|=\rho e^{-t}\le \rho e^{L} \ll 1\]
by our discussion in section $2$.

We want to apply Lemma \ref{lem.tangential estimate} and we recall that due to Corollary \ref{cor.quantitative no neck} we have $|\nabla v|(t, \theta) \lesssim e^{\frac{1}{10}(|t|-L)}\lesssim 1$.
Hence we estimate 
\begin{align*}
    \int_{[-L,L]\times T^1} |f|^2 &\lesssim \int_{[-L,L]\times T^1}  (|\nabla \tilde{X}|^2 + |\nabla v|^2 |\tilde{X}|^2) + \int_{[-L,L]\times T^1}|\nabla n|^4 |\tilde{X}|^2 \\&\lesssim (1+|\lambda|) \langle \langle X,X\rangle \rangle_u\,.
\end{align*}
Furthermore, we get
\begin{align*}
    \int_{[-L,L]\times T^1} (L-|t|) |f| \lesssim& \int_{[-L,L]\times T^1} (L-|t|)\left(|\nabla v|\left(|\nabla \tilde{X}| + |\nabla v| |\tilde{X}|\right) + |\nabla n|^2 |\tilde{X}| \right) \\
    \le& \left(\int_{[-L,L]\times T^1} (L-|t|)^2 |\nabla v|^2\right)^\frac12 (\langle X,X\rangle_u)^\frac12 \\
    &+ \left(\int_{[-L,L]\times T^1} (L-|t|)^2 |\nabla n|^2\right)^\frac12  \left(\int_{M} |X|^2\right)^\frac12  
\end{align*}
The right hand side is bounded since
\begin{align*}
    \int_{[-L,L]\times T^1} (L-|t|)^2 |\nabla v|^2 \lesssim \int_{[-L,L]\times T^1} (L-|t|)^2 e^{\frac15(|t|-L)}\lesssim 1,\\
    \int_{[-L,L]\times T^1} (L-|t|)^2 |\nabla n|^2 \le 4\pi \int_{0}^L \rho^2 (L-t)^2 e^{2t} \, dt \le 4\pi \rho^2 e^{2L} \int_{0}^\infty \tau^2 e^{-2\tau}\,d\tau \lesssim 1
\end{align*}
where we have used the definitions of $L$ and $\rho$ and the fact that $|\nabla n|^2=2\rho^2e^{-2t}$.

Furthermore, the convergence on the base and the bubbles, proposition \ref{prop.base}, \ref{prop.bubble} imply that $|\tilde{X}(\pm L, \theta)|$ is uniformly bounded. Hence every term on the right hand side of \eqref{eq.Linftybound} is uniformly bounded implying the desired $L^\infty$ bound for $X_k\circ n_k^i$. 
\medskip

Next, we test the eigenvalue equation \eqref{eq.eigenfunctionmoebius} with $\eta(t)^2 \tilde{X}$ where $\eta \equiv 1$ on $[-L+2l, L-2l]$ and its support is contained in $[-L+l, L-l]$, i.e. we may assume that $|\eta'|\lesssim l^{-1}$. We obtain the bound 
\begin{align*}
    \int \eta^2 \left( |\nabla X|^2 +\langle A^2_u(X),X\rangle \right) &\lesssim \int \eta^2 \left( |\nabla \tilde{X}|^2 + |\nabla v|^2 |\tilde{X}|^2\right)\\
     &\lesssim \int |\eta'|^2 |\tilde{X}|^2 + c \eta^2 \left(|\nabla v|^2 + |\nabla n|^2 \right) |\tilde{X}|^2\\
    &\lesssim \left( l^{-1} + e^{-\frac15l} + \rho^2 e^{2L} \right)\,.
\end{align*}
Therefore the claimed convergence \eqref{eq.quantitive convergence neck} follows immediately. (One might even use the equation once more to improve the convergence to a convergence in $C^{2,\alpha}$ but this is irrelevant for our further analysis.) 
\end{proof}

\subsection{Combining the different regimes and the proof of Theorem \ref{main2}}
We consider two arbitrary sequences of eigenfunctions $X_k, \hat{X}_k$ with corresponding eigenvalues $\lambda_k, \hat{\lambda}_k$, that might even agree. 
As a consequence of the previous sections, we can pass to an appropriate sub-sequence such that
\begin{align*}
    X_k &\xrightarrow{\text{``bubble converge''}} Y_0 \cup \{Z_i\}\\
    \hat{X}_k &\xrightarrow{\text{``bubble converge''}} \hat{Y}_0 \cup \{\hat{Z}_i\}\,
\end{align*}
where ``bubble converge'' is a short hand notation for the convergence established in the previous sections. 
In particular we can combine \eqref{eq.quantitive convergence base}, \eqref{eq.quantitive convergence bubble}, \eqref{eq.quantitive convergence neck}: 
Firstly due to \eqref{convergence base} we have 
\[\lim_{k \to \infty} \int_{M} |X_k + \hat{X}_k|^2 - |X_k|^2 - |\hat{X}_k|^2 = \int_M |Y_0+\hat{Y}_0|^2 - |Y_0|^2-|\hat{Y}_0|^2\,.\]
Secondly we split the part involving $A^2_{u_k}$ in the different regimes. For the sake of readability we introduce the quantity 
\[a(u, X, \hat{X}) = \langle A^2_{u}(X+\hat{X}), X+ \hat{X}\rangle -  \langle A^2_{u}(X), X\rangle - \langle A^2_{u}(\hat{X}),  \hat{X}\rangle\, \]
and estimate 
\begin{align}\label{eq.polarisationconvergence}
    \lim_{k\to \infty} \int_M & a(u_k, X_k, \hat{X}_k)= \lim_{r\downarrow 0}\lim_{k \to \infty } \Big(\int_{M \cup \{\dist(\cdot \Sigma_0)>r\}}a(u_k, X_k, \hat{X}_k)\nonumber\\
    &+ \sum_{i} \lim_{k \to \infty } \int_{\C \cup \{\dist(\cdot \Sigma_i)>r\}\cap B_{\frac1r}}a(u_k\circ m_k^i, X_k\circ m_k^i, \hat{X}_k\circ m_k^i) \nonumber\\
    &+ \sum_{i} \int_{[-L_{k}^i+M_k^i(r), L_k^i-M_k^i(r)]\times T^1} a(u_k\circ n_k^i, X_k\circ n_k^i, \hat{X}_k\circ n_k^i)\Big)\nonumber \\&
    = \int_M a(u_0, Y_0, \hat{Y}_0) + \sum_{i} \int_{\C} a(\omega_i, Z_i \hat{Z}_i) + 0
\end{align}
where the $M_k^i(r)\to \infty$ as $r\downarrow 0$ are chosen appropriately. 
The polarisation identity implies now the convergence of the scalar products.

\begin{proof}[Proof of Theorem \ref{main2}]
Let us denote by $Y_k(\lambda)$ the eigenspace of the index form  $D^2E(u_k)$ with respect to the scalar product $\langle\langle \cdot,\cdot \rangle\rangle_{u_k}$ (compare appendix \ref{sec.appA}), with respect to the eigenvalue $\lambda$. In particular we have 
\[\Ind(u_k)= \dim( \bigcup_{\lambda <0} Y_k(\lambda)) \text{ and } \Nul(u_k) = \dim Y_k(0)\,.\]
For a fixed $m \le \limsup_{k \to \infty} \Ind(u_k)+ \Nul(u_k)$ we may choose for every $k$ an orthogonal family of eigenfunctions $\{ X_k^\alpha\}_{\alpha=1}^m $. For each fixed $\alpha$ we can apply theorem \ref{thm.convergenceeigenfunctions} such that up to a subsequence 
\[ X^\alpha_k \xrightarrow{\text{``bubble converge''}} Y_0^\alpha \cup \{ Z_i^\alpha\}\,.\]
Due to the just established convergence \eqref{eq.polarisationconvergence} their union on the right build an $m$-dimensional space on which the index form 
\[ D^2E(u_0) + \sum_i D^2E(\omega_i)\]
is non positive. This fact immediately proves the theorem.
\end{proof}
\section{The case of general two-dimensional conformally invariant variational problems}
General conformally invariant integrands on a Riemann surface $M$ have been classified by Gr\"uter \cite{grueter} and they are given by 
\begin{equation}\label{eq.conformalinvarinatintegrand}
	\mathcal{E}(u) = \frac12\int_{M} |\nabla u|^2 \, dv_g + u^*\varpi 
\end{equation}
under the constraint that $u(x) \in N$ a.e. and where $\varpi$ is an arbitrary two-form on $N$. We will assume as before that $N$ is at least of class $C^3$ and that it is isometrically embedded into some euclidean space $\R^n$. Moreover, we assume that the $2$-form $\varpi$ is as well of class $C^3$ and without loss of generally we assume that it is globally defined on $\R^n$. 
By the celebrated result of Rivi\`ere \cite{riviere} on conservation laws for critical points of conformally invariant variational problems, the Euler-Lagrange equation of \eqref{eq.conformalinvarinatintegrand} can be written as
\[ \Delta u = \Omega(u)\cdot\nabla u\,, \]
where $\Omega(u) \in L^2(M, so(n)\otimes \bigwedge^1\R^n)$. More precisely one has the explicit expression
\begin{align}\label{eq.structureofOmega}
	\langle \xi, \Omega(u)\cdot \nabla u \rangle =& \langle \xi, A(u)(\nabla u, \nabla u)\rangle -\langle \nabla u, A(u)(\nabla u, \xi)\rangle \\\nonumber &+ d\varpi(u)(P(u)\xi\wedge P(u)\partial_1u\wedge P(u)\partial_2u)\,.
\end{align}
In particular the $\varepsilon$-regularity result of Rivi\`{e}re \cite{riviere} applies and one has the same bubbling phenomenon as in the previously considered case of harmonic map (see also \cite{lammsharp} for further details on this). 

The second variation for $\mathcal{E}$ is a bilinear form defined for for $X \in u^*TN$ of the following form, see e.g. \cite{dalio} for a derivation of this formula,
\begin{equation}\label{eq.generalindexform}
	D^2\mathcal{E}(u)(X)= \int_{M} (|\nabla X|^2 - \langle A^2_u(X),X\rangle + b_u(\nabla X, X) + c_u(X,X))\, dv_g\,
\end{equation} 
where the bilinear forms $b_u, c_u$ are determined by the geometric quantities of $N, \varpi$ and they satisfy the pointwise bounds 
\begin{equation}\label{eq.boundsoncoefficients}
	|b_u|\lesssim |\nabla u| \text{ and } |c_u|\lesssim |\nabla u|^2.
\end{equation}
In the following we want to give a short outline why the analysis presented for the harmonic map case directly applies to the general case of critical points of conformally invariant variational integrals since all relevant structural properties are preserved. 
\medskip

\emph{1. Universal $W^{1,2}$-bound and universal lower bound on the eigenvalues}
\medskip

Note the the growth conditions \eqref{eq.boundsoncoefficients} imply that for constants $0<c_1,c_2<$, depending only on $N$ and $\varpi$, one has 
\[ D^2\mathcal{E}(X,X) \le \frac12 \int_M |\nabla X|^2 dv_g- c_1 \int_{M} |\nabla u|^2 |X|^2 dv_g \le  \frac12 \int_M |\nabla X|^2 dv_g - c_2 \langle X,X\rangle_u\,. \]
By a small modification of the arguments presented in Proposition \ref{energybound} one also obtains a lower bound on the eigenvalues and a global $W^{1,2}$-bound for eigenfunctions of $D^2\mathcal{E}(u)$ with respect to the scalar product $\langle \langle \cdot, \cdot \rangle \rangle_u$.   
\medskip

\emph{2. Exponential decay of $\nabla u$ on the neck}
\medskip

First we note that due to \eqref{eq.structureofOmega} one has for any $i=1,2$ the identity 
\[ \langle \partial_i u, \Omega(u)\nabla u\rangle =0\,.\]
Hence one immediately checks that $\partial_{\bar{z}} (\partial_z u)^2 =0$, or equivalently, that the quadratic differential $(\partial_zv)^2 dz \otimes dz$ is holomorphic. 
Together with the fact that $|\Omega(u)\cdot \nabla u|\lesssim |\nabla u|^2$ one notes that Corollary \ref{cor.quantitative no neck} remains true. 
\medskip

\emph{3. Structure of the eigenvalue equation} 
\medskip

The eigenvalue equation for $D^2\mathcal{E}(u)$ determined by the scalar product $\langle\langle \cdot, \cdot \rangle \rangle_u$ is given by 
\begin{equation}\label{eq.generalEquationEigenvector1}P(u)(-\Delta X - A^2_u(X) - \operatorname{div}(\hat{b}^t_u(X)) + \hat{b}_u(\nabla X) + \hat{c}_u(X) )= \lambda (X+ A^2_u(X))\end{equation}
and can be rewritten in the form 
\[-\Delta X - \partial_i( (B^i_u)^t X)+  B_u^i \partial_i X + C_u X = \lambda X \]
where the coefficients $B_u^i$ and $C_u$ are determined by geometric quantities of $N, \varpi$ and evaluated at $u$. They satisfy the bounds
\[|B_u^i|\lesssim |\nabla u| \text{ and } |C_u|\lesssim (1+|\lambda|) |\nabla u|^2\,.
\]
As in the harmonic map case the eigenvalue equation \ref{eq.generalEquationEigenvector1} is well behaved when composing with a conformal transformation $m$: 
Namely, the transformed eigenfunction equation is given by 
\begin{align}\nonumber
    P(u\circ m) (-\Delta (X\circ m))&=(1+\lambda) A^2_{u\circ m} (X\circ m)+\frac\lambda2 |\nabla m|^2 X\circ m\\& + \operatorname{div}(\hat{b}^t_{u\circ m}(X\circ m)) + \hat{b}_{u\circ m}(\nabla (X\circ m)) + \hat{c}_{u\circ m}(X\circ m). \label{eq.generaleigenfunctionmoebius} 
\end{align}
From this we conclude that all the estimates for the eigenfunctions which we derived in the case of harmonic maps also remain true for the general case of critical points of conformally invariant variational problems.

\appendix
\section{Index considerations}\label{sec.appA}
To be self-contained we present Sylvester's law of inertia in our setting.
\medskip

For this we let $H$ be a Hilbert space and $V$ is an arbitrary vector space. Moreover, we assume that 
\begin{itemize}
    \item[i)] the embedding $H\hookrightarrow V$ is compact and dense,
    \item[ii)] $a:H\times H\to \R$ is a continuous and symmetric bilinear form
    \item[iii)] $\langle \cdot, \cdot,\rangle_S:V\times V\to \R$ is a scalar product on $V$ and there exists $\lambda>0$ so that 
    \[
    a(\cdot,\cdot)+\lambda \langle \cdot, \cdot \rangle_S
    \]
    is coercive, i.e. there exists $\alpha>0$ so that
    \[
    a(x,x)+\lambda \langle x,x\rangle_S \geq \alpha \|x\|_H^2
    \]
    for all $x\in H$.
\end{itemize}
We remark that $a$ induces an operator $A:H\to H^*$ via $a(x,y)=\langle Ax,y\rangle_H$ where $\langle \cdot,\cdot \rangle_H$ denotes the given scalar product on $H$. Similarly $\langle \cdot, \cdot \rangle_S$ induces an operator $I_S:H\to H^*$ on $H$.

It follows from the Lax-Milgram theorem that the map
\[
A+\lambda I_S:H\to H^*
\]
is a continuous bijection and thus we obtain a map $T:V\to V$ via 
\[
V\xrightarrow{i_S} H^* \xrightarrow{(A+\lambda I_S)^{-1}} H \hookrightarrow V,
\]
where $i_S$ is defined via $( i_S (v),x)=\langle v,x\rangle_S$ for all $x\in H$, $v\in V$ and $(\cdot,\cdot)$ denotes the dual pairing.

By assumption $i)$ the map $T$ is compact and thus $\sigma(T)$ is discrete with $0$ as the only possible accumulation point. In particular, we have a basis of eigenfunctions corresponding to the discrete set of eigenvalues.

We next observe that $T$ is symmetric, since for $v,w\in H$ and $x,y\in H$ so that $(A+\lambda I_S)x=v$ respectively $(A+\lambda I_s)y=w$ we have
\[
\langle Tv,w\rangle_S =\langle x,w\rangle_S =\langle x,(A+\lambda I_S)y\rangle_S =\langle (A+\lambda I_S)x, y\rangle_S=\langle v,y\rangle_S =\langle v,Tw\rangle_S.
\]
Next we introduce the operator $A_S:=(T^{-1}-\lambda id)$ which is densely defined on $V$ with domain $i(H)$, which corresponds to $i_S^* \circ A$ where $i_S$ is the inclusion of $V$ into $H^*$ via $\langle \cdot, \cdot \rangle_S$. This follows since for $x\in H$ we let $y$ be defined via $y=T^{-1}\circ i(x)$. Hence $i(x)=T(y)$ resp. $x=(A+\lambda I_S)^{-1}\circ i_S(y)$. This is equivalent to $(A+\lambda I_S)x=i_S(y)$ and thus
\[
y-\lambda i(x)=i_S^*\circ i_S(y)-\lambda i_S^*\circ I_S(x)=i_S^* Ax.
\]
Here we used that $(i_S^*\circ i_S)=id$.

We note that $A_S$ is not continuous and we have
\[
\langle A_Sx,y\rangle_S=\langle A_Sy,x\rangle_S
\]
for all $x,y\in H$.

The following Lemma characterizes the nullity and the index of the operators which we just defined.
\begin{lemma}\label{char.index}
With the notations from above we have the following characterizations:
\begin{itemize}
    \item[i)] $\cup_{\lambda<0} \operatorname{ker}(A_S-\lambda id)=\cup \{W\subset H:\, a|_W<0\}=W_{max}$
    \item[ii)] $\operatorname{ker} (A_S)=\{x\in H:\, a(x,y)=0\ \ \ \forall y\in H\}$.
\end{itemize}
\end{lemma}
\begin{proof}
The statement $ii)$ is obvious since we have the identity
\[
( i_S^*\circ A x,y)=a(x,y)=(Ax,y)
\]
for all $x,y\in H$.
\newline
Now we show $i)$. For this we note that we can write
\[
\sigma(A_S)= \{\lambda_i:\, i\in \N\}=\| \frac{1}{\mu}-\lambda:\, \mu \in \sigma(T)\}.
\]
We know that $\lambda_i\in \R$, $\lambda_i<\lambda_{i+1}$ for all $i\in \N$ and $\lim_{i\to \infty} \lambda_i = \infty$. If $x\in \cup_{\lambda<0} \text{ker} (A_S-\lambda id)=:\hat{W}$, then it follows from the above that there exists $N\in \N$ so that
\[
x=\sum_{j=1}^N x_j,
\]
where $x_j\in \text{ker}(A_S-\lambda_j id)$. Thus we calculate
\begin{align*}
    a(x,x)=& \sum_{j=1}^N a(x_j,x)=\sum_{j=1}^N \langle A_S x_j, x\rangle_S\\
    =& \sum_{j=1}^N \lambda_j \langle x_j,x\rangle_S =\sum_{\lambda=1}^N \lambda_j \|x_j\|^2<0
\end{align*}
and thus $a|_{\hat{W}\otimes \hat{W}}$ is negative definit.

This shows that $\hat{W}\subset W_{max}$. 

For the other inclusion we let $x\in W_{max} \cap \hat{W}^{\perp} \subset \cup_{j>N} \text{ker} (A_S-\lambda_j id)=\cup_{\lambda\geq 0} \text{ker} )(A_S-\lambda id)$. Thus we get
\begin{align*}
    a(x,x) =\sum_{j>N} \langle A_Sx_j,x\rangle_S \geq \sum_{j>N} \lambda_j \|x_j\|^2 \geq 0
\end{align*}
but since also $x\in W_{max}$ we get $a(x,x)\leq 0$ and hence $x=0$.
\end{proof}
\begin{remark}
\begin{itemize}
\item[1)] Note that it follows from the above characterization that the index and the nullity are independent of the underlying vector space and the scalar product.
\item[2)] In our case we consider the symmetric bilinear form associated to the second variation of the energy at a critical harmonic map $u\colon M \to N$ along a variation $X \in u^*TN$ i.e.
\[\frac{d^2}{dt^2}E(u_t)=D^2E(u)[X,X]= \int_{M} (|\nabla X|^2 - \langle A_u^2(X), X\rangle) \, dv_g\;.\]
Hence if we choose $\langle\langle \cdot, \cdot \rangle \rangle_u$ the weak form of the eigenvalue equation is given by $D^2E(u)[X,Y]=\lambda \langle \langle X,Y \rangle \rangle_u$ for all $Y \in C^1(M, u^*TN)$. Written in integral form and as a PDE it reads 
\begin{align}
    \int_{M} \langle (\nabla X, \nabla Y\rangle - \langle A_u^2(X), Y\rangle) \, dv_g &= \lambda \int_{M} \langle X + A_u^2(X), Y\rangle \, dv_g \nonumber\\
    -P(u)\Delta X - A_u^2(X) &= \lambda (X + A_u^2(X)) \label{eq.eignevalueequation}
\end{align}
where $P(p)\colon \R^n \to T_pN$ denotes the orthogonal projection onto the tangent space of $N$ at the point $p\in M$.
\end{itemize}
\end{remark}


\begin{thebibliography}{10}

\bibitem{chen99}
J.~Chen and G.~Tian.
\newblock Compactification of moduli space of harmonic mappings.
\newblock {\em Comm. Math. Helv.}, 74:201--237, 1999.

\bibitem{chenliwang}
L.~Chen, Y.~Li and Y.~Wang.
\newblock The refined analysis on the convergence behavior of harmonic map sequence from cylinders. 
\newblock {\em J. Geom. Anal.} 22:942--963, 2012. 

\bibitem{chodoshmantoulidis}
O.~Chodosh and C.~Mantoulidis.
\newblock Minimal surfaces and the Allen-Cahn equation on 3-manifolds: index, multiplicity, and curvature estimates. 
\newblock {\em Ann. of Math. (2)}, 191:213--328, 2020. 

\bibitem{dalio}
F.~Da Lio, M.~Gianocca and T.~Rivi\`ere.
\newblock Morse index stability for critical points to conformally invariant Lagrangians.
\newblock arXiv:2212.03124, 2022.

\bibitem{ding95}
W.~Y. Ding and G.~Tian.
\newblock Energy identity for a class of approximate harmonic maps from
  surfaces.
\newblock {\em Comm. Anal. Geom.}, 3:543--554, 1995.

\bibitem{grueter}
M.~Gr\"uter.
\newblock Conformally invariant variational integrals and the removability of isolated singularities. 
\newblock {\em Manuscripta Math.} 47:85--104, 1984. 

\bibitem{helein02}
F.~H{\'e}lein.
\newblock {\em Harmonic maps, conservation laws and moving frames}, volume 150
  of {\em Cambridge Tracts in Mathematics}.
\newblock Cambridge University Press, Cambridge, 2002.

\bibitem{stern}
M.~Karpukhin and D.~Stern.
\newblock Min-max harmonic maps and a new characterization of conformal eigenvalues.
\newblock arXiv:2004.04086, to appear in {\em J. Eur. Math. Soc.}

\bibitem{lammsharp}
T.~Lamm and B.~Sharp.
\newblock Global estimates and energy identities for elliptic systems with antisymmetric potentials. 
\newblock {\em Comm. Partial Differential Equations}, 41:579--608, 2016. 


\bibitem{lin98}
F.~Lin and C.~Wang.
\newblock Energy identity of harmonic map flows from surfaces at finite
  singular time.
\newblock {\em Calc. Var. Partial Differ. Equ.}, 6:369--380, 1998.

\bibitem{lin99}
F.~Lin and C.~Wang.
\newblock Harmonic and quasi-harmonic spheres.
\newblock {\em Comm. Anal. Geom.}, 7:397--429, 1999.

\bibitem{lin02}
F.~Lin and C.~Wang.
\newblock Harmonic and quasi-harmonic spheres {II}.
\newblock {\em Comm. Anal. Geom.}, 10:341--375, 2002.

\bibitem{marquesneves}
F.~Marques and A.~Neves.
\newblock Morse index of multiplicity one min-max minimal hypersurfaces. 
\newblock {Adv. Math.} 378, Paper No. 107527, 2021. 

\bibitem{micallefmoore}
M.~Micallef and J.~Moore.
\newblock Minimal two-spheres and the topology of manifolds with positive curvature on totally isotropic two-planes. 
\newblock {\em Ann. of Math. (2)} 127:199--227, 1988. 

\bibitem{parker96}
T.~Parker.
\newblock Bubble tree convergence for harmonic maps.
\newblock {\em J. Differ. Geom.}, 44:595--633, 1996.

\bibitem{qing95}
J.~Qing.
\newblock On singularities of the heat flow for harmonic maps from surfaces into spheres.
\newblock {\em Comm. Anal. Geom.} 3:297--315, 1995. 

\bibitem{qing97}
J.~Qing and G.~Tian.
\newblock Bubbling of the heat flow for harmonic maps from surfaces.
\newblock {\em Comm. Pure Appl. Math.}, 50:295--310, 1997.

\bibitem{riviere}
T.~Rivi\`ere.
\newblock Conservation laws for conformally invariant variational problems.
\newblock {\em Invent. Math.}, 168:1--22, 2007. 

\bibitem{sacks81}
J.~Sacks and K.~Uhlenbeck.
\newblock The existence of minimal immersions of $2$-spheres.
\newblock {\em Annals of Math.}, 113:1--24, 1981.

\bibitem{sun1}
Y.~Sun
\newblock Morse index bound for minimal two spheres.  
\newblock {\em J. Geom. Anal.} 32:3, Paper No. 72, 2022. 

\bibitem{sun2}
Y.~Sun
\newblock Morse index bound for minimal torus.  
\newblock arXiv:2111.06922, 2021. 

\bibitem{yin1}
H.~Yin.
\newblock Generalized neck analysis of harmonic maps from surfaces. 
\newblock {\em Calc. Var. Partial Differential Equations} 60:3, Paper No. 117, 2021. 

\bibitem{yin2}
H.~Yin.
\newblock Higher-order neck analysis of harmonic maps and its applications.
\newblock {\em Ann. Global Anal. Geom.} 62:457--477, 2022. 
\end{thebibliography}
\end{document}